\documentclass{amsart}

\usepackage{amsmath, amssymb}

\usepackage{graphicx,epsfig}

\theoremstyle{plain}
\newtheorem{thm}{Theorem}[section]

\theoremstyle{definition}

\begin{document}

\title{The Homogeneous Property of the Hilbert Cube}

\author{Denise M. Halverson and David G. Wright}

\begin{abstract}
We demonstrate the homogeneity of the Hilbert Cube.  In particular, we construct explicit self-homeomorphisms of the Hilbert cube so that given any two points, a homeomorphism moving one to the other may be realized.
\end{abstract}

\subjclass[2010]{Primary 57N20; Secondary 54B10}

\maketitle

\section{Introduction}

It is well-known that the Hilbert cube is homogeneous, but proofs such as those in Chapman's CBMS Lecture Notes \cite{Chapman} show the existence of a homeomorphism rather than an explicit homeomorphism.  The purpose of this paper is to demonstrate how to construct explicit self-homeomorphisms of the Hilbert cube so that given any two points, a homeomorphism moving one to the other may be realized.  This is a remarkable fact because in a finite dimensional manifold with boundary, interior points are topologically distinct from boundary points.  Hence, the $n$-dimensional cube $$C^n = \overset{n}{\underset{i=1}{\Pi}} I_i,$$ where $I_i = [-1,1]$, fails to be homogeneous because there is no homeomorphism of $C^n$ onto itself mapping a boundary point to an interior point and vice versa.  Thus, the fact that the Hilbert cube is homogeneous implies that there is no topological distinction between what we would naturally define to be the boundary points and the interior points, based on the product structure.

Formally, the \emph{Hilbert cube} is defined as $$Q = \overset{\infty}{\underset{i=1}{\Pi}} I_i$$ where each $I_i$ is the closed interval $[-1,1]$ and $Q$ has the product topology.  A point $p \in Q$ will be represented as $p=(p_i)$ where $p_i \in I_i$.  A metric on $Q$ is given by $$d(p,q) = \underset{i=0}{\overset{\infty}{\Sigma}} \dfrac{|p_i - q_i|}{2^i}$$ where $p,q \in Q$.  The \emph{pseudo-interior} of $Q$ is given by $$Q^o = \overset{\infty}{\underset{i=1}{\Pi}} I_i^o$$ where each $I_i^o$ is the open interval $(-1,1)$.  The \emph{pseudo-boundary} of $Q$ is denoted $B(Q)$ and is defined as $B(Q)=Q-Q^o$.  Since the arbitrary product of compact spaces with the product topology is compact, we note that $Q$ must be compact.

It is an easy matter to show that given any two points $p, q \in Q^o$, there is a self-homeomorphism of $Q$ taking $p$ to $q$.  Such a homeomorphism may be constructed by defining on each coordinate $I_i$ of $Q$ a piecewise linear map of $I_i$ onto itself fixing the boundary points and taking $p_i$ onto $q_i$.  In particular we define $f_i: I_i \to I_i$;
$$f_i(t) = \left\{
                   \begin{array}{ll} \medskip
                     (t+1)\left( \dfrac{q_i+1}{p_i+1} \right)-1, & \hbox{ if } -1 \leq t \leq p_i;\\
                     (t-p_i)\left( \dfrac{1-q_i}{1-p_i} \right)+q_i, & \hbox{ if } p_i < t \leq 1.
                   \end{array}
                 \right.$$
where $p = (p_i)$ and $q = (q_i)$.  Then let $f: Q \to Q$;
\begin{equation} \label{E1}
f(x) = (f_i(x_i))
\end{equation}
for any $x \in Q$.  Then $f(p) = q$.  Since $f$ is the product of homeomorphisms, clearly, any basic open set in the product topology is mapped to a basic open set by both $f$ and $f^{-1}$.  Thus both $f$ and $f^{-1}$ are continuous bijective mappings of $Q$ onto itself so $f$ is a homeomorphism.  Therefore, it only remains to be shown that given any point in the pseudo-boundary, there is a homeomorphism taking it to a point in the pseudo-interior.

\section{General Definitions}

An arbirary product space $Y^J$ is defined as $$Y^J = \underset{\alpha \in J}{\Pi} Y_{\alpha}$$ where each $Y_{\alpha}$ is a copy of $Y$.  Also, $(X,d)$ denotes a metric space with metric $d$.  Note that if $X$ is a compact space then $(X,d)$ must be a bounded and complete metric space.  The set of all continuous functions of $X$ into $Y$ is denoted $C(X,Y)$, or as $C(X)$ if $X$ and $Y$ are the same space.  Then $H(X)$ represents the set of all homeomorphisms in $C(X)$.  Given a metric space $(Y,d)$ and product space $Y^J$, the \emph{sup-norm metric} is defined as $\rho: Y^J \times Y^J \to \mathbb{R}$; $$\rho(f,g) = \sup\{d(f(x),g(x)) \ | \ x \in J \}.$$

\section{A First Attempt}

Consider the pseudo-boundary point $p=(1,1,1,\ldots)$.  A first attempt at the construction of a homeomorphism that maps $p$ into the pseudo-interior may be to consider the homeomorphism $f_n^m$ wich isometrically twists the square boundary of the $2$-cell $I_n \times I_m$ of $Q$ a distance of one unit in the counter-clockwise direction and then maps the interior points by radial extension (Figure \ref{F1}).

\begin{figure}
  \centering
  \includegraphics[scale=0.7]{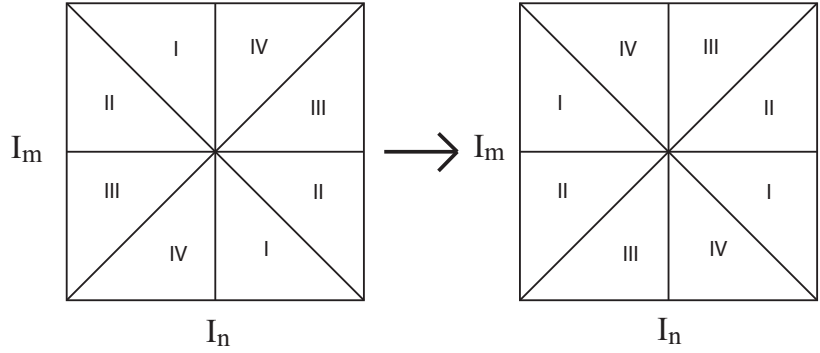}
  \caption{A guide to the mapping $f_n^m$.}
  \label{F1}
\end{figure}

Formally, we define $\overline{f_n^m}: I_n \times I_m \to I_n \times I_m$;
$$\overline{f_n^m}(x,y)=\left\{
                     \begin{array}{ll}
                       (-y,x+y) \text{ if } |x| \leq |y| \text{ and } xy<0; \\
                       (x,x+y) \text{ if } |x| \geq |y| \text{ and } xy<0;  \\
                       (x-y,x) \text{ if } |x| \geq |y| \text{ and } xy \geq 0;  \\
                       (x-y,y) \text{ if } |x| \leq |y| \text{ and } xy \geq 0.
                     \end{array}
                   \right. $$ where $x \in I_n$ and $y \in I_m$.  Let $f_n^m$ be defined by crossing $\overline{f_n^m}$ with the identity mapping on all other factors.  Then $$f_1^2(p) = (0,1,1,\ldots).$$ In general, $$f_n^{n+1} f_{n-1}^n \ldots f_1^2(p) = (0,0, \ldots 0,1,1,\ldots)$$ where the first $n$ coordinates are $0$.  Let $$f = \underset{n \to \infty}{\lim} f_n^{n+1} f_{n-1}^n \ldots f_1^2.$$   It may be observed that $\left\{ f_n^{n+1} f_{n-1}^n \ldots f_1^2 \right\}$ is a Cauchy sequence in $\rho$ since $$\rho\left( f_{n+1}^{n+2}\left(f_n^{n+1} f_{n-1}^n \ldots f_1^2\right), f_n^{n+1} f_{n-1}^n \ldots f_1^2 \right) \leq \dfrac{2}{2^{n+1}} + \dfrac{2}{2^{n+2}} < \dfrac{1}{2^{n-1}}.$$  It will be shown shortly that the set of all continuous functions on a complact metric space is complete with respect to $\rho$ so we will be able to conclude that $f$ exists and is continuous.  Thus $f$ is a continuous function which maps $p$ to the origin of the Hilbert cube.

However, although $f$ is continous, $f$ is not a homeomorphism.  This is demonstrated by considering any point $q = (t.t.t.\ldots)$ where $-1 \leq t < 1$.  Observe that $$f_1^2(q) = (0,t,t,\ldots)$$ and in general, $$f_n^{n+1} f_{n-1}^n \ldots f_1^2(q) = (0,0, \ldots 0,t,t,\ldots)$$ where the first $n$ coordinates are $0$.  Thus $f(q) = 0$.  Therefore $f$ is not one-to-one, so $f$ is not a homeomorphism.

\section{Convergence Criteria}

We would like to establish a criteria which will guarantee that a sequence of compositions of homeomorphisms converges to a homeomorphism.  First it will be shown that the metric space $(\mathcal{C}(X), \rho)$ is complete and then we will extend to a more useful topologically equivalent metric which is complete in $H(X)$.

\begin{thm} \label{T1}
Let $(Y,d)$ be a compact metric space.  Then $(Y^J,\rho)$ is a complete metric space for any set $J$ where $\rho$ is the sup-norm metric.
\end{thm}

\begin{proof}
First, note that $(Y,d)$ is a complete metric space since $Y$ is compact.  Let $\{f_i\}$ be a Cauchy sequence in $Y^J$.  Let $\alpha \in J$.  Then $$d\left(f_n(\alpha), f_m(\alpha)\right) \leq \rho(f_n, f_m).$$  Thus $\{f_i\}$ is a Cauchy sequence in $(Y,d)$ and therefore must converge to some point $y_{\alpha} \in Y$.  Define $f: J \to Y; f(\alpha) = y_{\alpha}$.  We want to show that $\{f_i\} \to f$.

Since $\{f_i\}$ is a Cauchy sequence with respect to $\rho$, then for any $\epsilon > 0$, there exists a number $N$ such that whenever $n,m > N$, then $$\rho\left(f_n,f_m \right) < \frac{\epsilon}{2}.$$ Thus, $$d\left(f_n(\alpha)), f_m(\alpha)\right) < \frac{\epsilon}{2}$$ for all $\alpha$.  By holding $\alpha$ and $n$ fixed and letting $m$ go to infinity, we obtain $$d\left(f_n(\alpha), f(\alpha)\right) \leq
\frac{\epsilon}{2}.$$  This must be true for all $\alpha$ so $$\rho(f_n, f) = \sup \{ d(f_n(\alpha), f(\alpha)) \ | \ \alpha \in Y \} \leq \frac{\epsilon}{2} < \epsilon.$$  Then for $n > N$, $\rho \left(f_n, f\right) < \epsilon$.  Therefore $\{f_i\}\to f$ so $\left(Y^J, \rho\right)$ is complete.
\end{proof}

\begin{thm} \label{T2}
Let $X$ be a topological space and $(Y,d)$ a compact metric space.  Then $(\mathcal{C}(X,Y), \rho)$ is complete where $\mathcal{C}(X,Y)$ is the set of all continuous functions from $X$ into $Y$ and $\rho$ is the sup-norm metric.
\end{thm}

\begin{proof}
Let $\{f_i\}$ be a Cauchy sequence in $(\mathcal{C}(X,Y), \rho)$. Then by Theorem \ref{T1} we know that $\{f_i\} \to f$ where $f \in Y^X$.  Then for any $\epsilon > 0$, there exists an $N$ such that whenever $n > N$ then $\rho(f_n,f) < \epsilon$.  Thus for all $x \in X$, $$d(f_n(x), f(x)) \leq \rho(f_n,f) < \epsilon.$$ Then $\{f_i\} \to f$ uniformly.  It is a well known fact that the limit function of a uniformly convergent sequence of continuous functions converges to a continuous function, so $f \in \mathcal{C}(X,Y)$.
\end{proof}

We now define a metric, $\zeta$, which will be more useful in establishing a convergence criteria for $H(X)$.  Let $\zeta: H(X) \times H(X) \to \mathbb{R}$; $$\zeta(f,g)= \rho(f,g) + \rho(f^{-1}, g^{-1})$$ where $f,g \in H(X)$.  It is easily checked that $\zeta$ satisfies all the properties of a metric.  The next theorem shows that $\zeta$ is an equivalent metric to $\rho$ which is complete in $H(X)$.

\begin{thm} \label{T3}
The metric $\zeta$ is equivalent to $\rho$ in $H(X)$ where $(X,\rho)$ is a compact metric space.  Furthermore, $(H(X), \zeta)$ is a complete metric space.
\end{thm}

\begin{proof}
In order to show that $\rho$ and $\zeta$ are equivalent, we must show that:
\begin{enumerate}
\item For any open set $V \subset (H(X), \rho)$ and any point $f \in H(X)$ such that $f \in V$, there exists an open set $U \subset (H(X), \zeta)$  such that $f \in U$ and $U \subset V$.
\item For any open set $U \subset (H(X), \zeta)$  and any point $f \in H(X)$ such that $f \in U$, there exists an open set $V \subset (H(X), \rho)$  such that $f \in V$ and $V \subset U$.
\end{enumerate}
To show (1), it suffices to show that a basic open set in $(H(X), \rho)$ centered about the point $f$ contains a basic open set in $(H(X), \zeta)$ centered about the same point.  A basic open set $V \subset (H(X), \rho)$ centered about $f \in H(X)$ is of the form $$V = \{ h \in H(X) \ | \ \rho(f,h) < \epsilon \}.$$  Let $U \subset \{H(X), \zeta \}$ be the basic open set centered about $f$ that is $$U = \{h\in H(X) \ | \ \zeta(f,h) < \epsilon \}.$$  Then for any $h \in U$, $$\rho(f,h) \leq \rho(f,h) + \rho(f^{-1}, h^{-1}) = \zeta(f,h) < \epsilon.$$  Thus $h$ is contained in $V$, so $U \subset V$.

To prove (2) is a bit more difficult.  Let $U$ be the same open set as defined above.  We want to find a basic open set, $W \subset (H(X), \rho)$, which contains $f$ and is contained in $U$.  First note that $H(X)$ is a set of uniformly continuous functions since $X$ is compact.  Since $f \in H(X)$, then $f^{-1} \in H(X)$.  Thus $f^{-1}$ is uniformly continuous.  By the uniform continuity of $f^{-1}$, there exists a $\delta > 0$ such that whenever $d(f(x), f(y)) < \delta$, then $d(x,y) < \frac{\epsilon}{2}$.  Without loss of generality we can choose $\delta < \frac{\epsilon}{2}$.  Then we let $W \subset (H(X), \rho)$ be a basic open set centered about $f$ such that $$W = \{h \in H(X) \ | \ \rho(f,h) < \delta \}.$$  Let $x \in X$ and $h \in W$.  Then $h(x) = f(z)$ for some unique $z \in X$ since both $h$ and $f$ are one-to-one and onto mappings of $X$ into $X$.  Since $d(f(x), h(x)) < \delta$, then $$d(f(x), f(z)) < \delta.$$  By uniform continuity $$d(x,z) < \frac{\epsilon}2.$$  Noting that $x = h^{-1}(h(x))$ and $z=f^{-1}(f(z))$ we rewrite this expression as $$d(h^{-1}(h(x)), f^{-1}(f(z)) < \frac{\epsilon}2$$ and then, $$d(h^{-1}(h(x)), f^{-1}(h(x))) < \frac{\epsilon}2.$$ Since $h(x)$ maps to all points in $X$,
\begin{align*}
\rho(h^{-1}, f^{-1}) =& \sup \{ d\left(h^{-1}(y), f^{-1}(y)\right) \ | \ y \in X \} \\ =& \sup \{ d\left(h^{-1}(h(x)), f^{-1}(h(x))\right) \ | \ x \in X \} \leq \frac{\epsilon}2.
\end{align*}
Therefore, $$\zeta(f,h) = \rho(f,h) + \rho\left(f^{-1},h^{-1}\right) < \frac{\epsilon}2 + \frac{\epsilon}2 = \epsilon.$$  Then $h$ is contained in $U$, so $W \subset U$.

Therefore $\rho$ and $\zeta$ are equivalent metrics.

Finally we need to show that $(H(X), \zeta)$ is a complete metric space.  Let $\{f_i\}$ be a Cauchy sequence in $(H(X), \zeta)$.  Then $\{f_i\}$ and $\{f^{-1}_i\}$ are Cauchy sequences in $(H(X), \rho)$. Since $\rho$ is complete in $\mathcal{C}(X)$, it follows from Theorem \ref{T2} that $\{f_i\} \to f$ and $\{f^{-1}_i\} \to g$ for some functions $f, g \in \mathcal{C}(X)$.  Let $f(x) = y$ where $x \in X$.  We need to show that $g(y) = x$.  By continuity of $f^{-1}$, for any $\epsilon > 0$, there exists $\delta > 0$ such that whenever $d\left(y, f_i(x)\right) < \delta$ then $d\left(f_i^{-1}(y), f^{-1}_i(f_i(x))\right) < \epsilon$.  By the convergence of $\{f_i\}$ to $f$, there exists an $N$ such that whenever $n \geq N$, then $d\left(y, f_i(x)\right)< \delta$.  In other words $d\left(f_i^{-1}(y), x\right) < \epsilon$ whenever $n \geq N$.  Hence we may conclude that $f_i^{-1}(y) \to x$, so $g(y) = x$. Therefore $g = f^{-1}$.
\end{proof}

We now give criteria for convergence of an infinite composition of a sequence of homeomorphisms to a homeomorphism.

\begin{thm} \label{T4}
Let $X$ be a compact space and for each integer $n \geq 1$ let $H_n$ be a family of homeomorphisms of $X$ onto itself such that for each $\epsilon > 0$ there exists an $f \in H_n$ such that $\rho(f, \text{id}) < \epsilon$.  Then we can select $h_n \in H_n$ such that the left composition $$h=\underset{n \to \infty}{\lim} h_nh_{n-1}\ldots h_1$$ defines a homeomorphism.  Furthermore, $h$ can be chosen so that $\rho(h,\text{id})$ is arbitrarily small.
\end{thm}

\begin{proof}
In the previous theorem it was shown that $\zeta$ and $\rho$ are equivalent metrics and that $(H(X), \zeta)$ is a complete metric space.  First let $\epsilon > 0$ be given.  Choose $h_1 \in H_1$ so that $\zeta(h_1, \text{id}) < \frac{\epsilon}2$.  Then it is also true that $\rho(h_1, \text{id}) < \frac{\epsilon}2$.  Inductively, choose $h_{n+1} \in H_{n+1}$ so that  $$\zeta(h_{n+1}h_n \ldots h_1, h_n \ldots h_1) < \frac{\epsilon}{2^{n+1}}.$$  Again, it is also true that $$\rho(h_{n+1}h_n \ldots h_1, h_n \ldots h_1) < \frac{\epsilon}{2^{n+1}}.$$  Thus $\{ h_nh_{n-1}\ldots h_1 \}$ is a Cauchy sequence in $(H(X), \zeta)$.  Then we know by Theorem \ref{T3} that $$h = \underset{n \to \infty}{\lim} h_n h_{n-1} \ldots h_1$$ exists and $h \in H(X)$.  Observe that $\rho(h, \text{id}) < \epsilon$.
\end{proof}

\section{Main Result}

It will now be demonstrated that a pseudo-boundary point of $Q$ can be mapped into the pseudo-interior of $Q$ by a homeomorphism. First we define a family of homeomorphisms from which homeomorphisms may be chosen arbitrarily close to the identity.  Let $\epsilon_j = \dfrac{1}{2^j}$.  Then $2\epsilon_j$ is the length of the $j$th coordinate space of the Hilbert cube with respect to the metric $d$.  For each $n > 0$ and $m > n$, define $\overline{\phi_n^m}: I_n \times I_m \to I_n \times I_m;$ $$\overline{\phi_n^m}\left( x,y \right) = \left\{
        \begin{array}{ll}
        \left( \sigma -  \frac{\epsilon_m}{\epsilon_n} (y + \sigma ), y+\sigma + \frac{\epsilon_n}{\epsilon_m} (x-\sigma)  \right) & \text{ if Type I;} \\
        \left( x, y + \sigma + \frac{\epsilon_n}{\epsilon_m}(x-\sigma)   \right) & \text{ if Type II;} \\
        \left( x  - y \dfrac{\sigma(1-  \frac{\epsilon_m}{\epsilon_n})-x}{\frac{\epsilon_n}{\epsilon_m}(x-\sigma) + \sigma},  \sigma + \frac{\epsilon_n}{\epsilon_m} (x-\sigma)  \right) & \text{ if Type III;} \\
        \left( x -  \frac{\epsilon_m}{\epsilon_n}y, y \right) & \text{ if Type IV.}
        \end{array}
        \right.$$
where $(x,y) \in I_n \times I_m$ and $$\sigma = \left\{
                                                  \begin{array}{ll}
                                                    1, & \text{ if } x \geq 0; \\
                                                    -1, & \text{ if } x < 0.
                                                  \end{array}
                                                \right.$$
Each point $(x,y)$ satisfies one of the following conditions:
\begin{align*}
&\text{Type I} & \quad & xy \leq 0,\  1-\frac{\epsilon_m}{\epsilon_n} \leq |x| \leq 1,  \text{ and } \frac{\epsilon_n}{\epsilon_m}(|x|-1)+1 \leq |y| \leq 1 \\
&\text{Type II} & \quad & xy \leq 0,\  1-\frac{\epsilon_m}{\epsilon_n} \leq |x| \leq 1,  \text{ and }  0\leq |y| \leq \frac{\epsilon_n}{\epsilon_m}(|x|-1) +1 \\
&\text{Type III} & \quad & xy \geq 0,\  1-\frac{\epsilon_m}{\epsilon_n} \leq |x| \leq 1,  \text{ and } 0\leq |y| \leq \frac{\epsilon_n}{\epsilon_m}(|x|-1) +1 \\
&\text{Type IV} & \quad & xy \geq 0,\  1-\frac{\epsilon_m}{\epsilon_n} \leq |x| \leq 1,  \text{ and } \frac{\epsilon_n}{\epsilon_m}(|x|-1)+1 \leq |y| \leq 1;  \text{ or } |x| \leq 1-\frac{\epsilon_m}{\epsilon_n}\\
\end{align*}
Essentially $\overline{\phi_n^m}$ first shrinks the $I_n \times I_m$ to an $\epsilon_nI_n \times \epsilon_mI_m$ cell.  The center line, which consists of points with coordinates $(x,y)$ such that $|x| \leq 1-\epsilon_m$ and $y=0$, is held fixed while the boundary is twisted isometrically in the counter-clockwise direction a distance of $\epsilon_m$.  This gives a $\frac{\pi}{4}$ twist of the vertical segments which intersect the center line.  Points with first coordinate such that $|x| > 1-\epsilon_m$ are mapped by radial extension about the endpoints of the centerline (Figure \ref{F2}).  Finally, the $\epsilon_nI_n \times \epsilon_mI_m$ cell is mapped back to the $I_n \times I_m$ cell.
\begin{figure}\label{F2}
  \centering \label{F2}
  \includegraphics[scale=0.7]{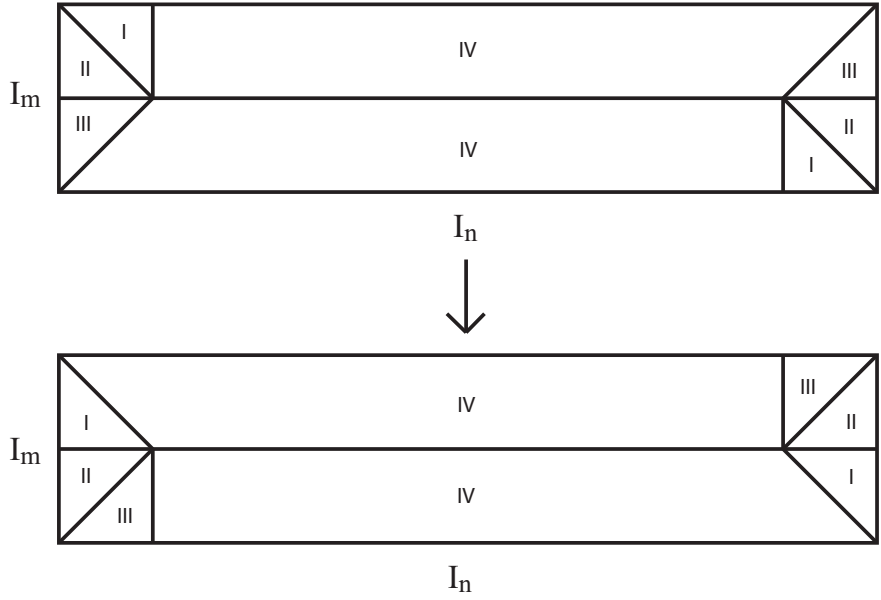}
  \caption{A guide to the mapping $\phi_n^m$.}
\end{figure}
Clearly, three consecutive twists on the $I_n \times I_m$ cell of this type move points with first coordinate on the boundary of $I_n$ to points with first coordinate in the interior of $I_n$.  In other words, $\overline{\left[\phi^m_n\right]}^3(\{-1,1\} \times I_m) \subset I_n^o \times \{-1,1\}$.  Define $\phi_n^m$ to be $\overline{\phi_n^m}$ crossed with the identity map on all other factors.  Observe that $\rho\left(\phi_n^m, \text{id}\right) \leq \epsilon_m$.  Let $\Phi_n^m = \left[\phi_n^m\right]^3$.  Then $\Phi_n^m$ removes points from the boundary of the $I_n$ cell and $\rho\left(\Phi_n^m\right) \leq 3 \epsilon_m$.

Similarly, the inverse map of each $\Phi_n^m$ is the homeomorphism $\Psi_n^m: Q \to Q$ defined as $\left[ \psi_n^m \right]^3$ where $\psi_n^m$ is defined by crossing $\overline{\psi_n^m}$ with the identity map on all other factors where $\overline{\psi_n^m}: I_n \times I_m \to I_n \times I_m$; $$\overline{\psi_n^m}\left( x,y \right) = \left\{
        \begin{array}{ll}
        \left( x  + y \dfrac{\sigma(1-  \frac{\epsilon_m}{\epsilon_n})-x}{\frac{\epsilon_n}{\epsilon_m}(x-\sigma) + \sigma}, -\sigma -\frac{\epsilon_n}{\epsilon_m} (x-\sigma)   \right) & \text{ if Type I';} \\
        \left( x, y - \sigma - \frac{\epsilon_n}{\epsilon_m}(x-\sigma)    \right) & \text{ if Type II'';} \\
        \left( \sigma -  \frac{\epsilon_m}{\epsilon_n} (-y + \sigma ), y-\sigma - \frac{\epsilon_n}{\epsilon_m} (x-\sigma)  \right) & \text{ if Type III';} \\
        \left( x + \frac{\epsilon_m}{\epsilon_n}y, y \right) & \text{ if Type IV'.}
        \end{array}
        \right.$$
The conditions are given as follow:
\begin{align*}
&\text{Type I'} & \quad & xy \leq 0,\  1-\frac{\epsilon_m}{\epsilon_n} \leq |x| \leq 1,  \text{ and }  0\leq |y| \leq \frac{\epsilon_n}{\epsilon_m}(|x|-1) +1 \\
&\text{Type II'} & \quad & xy \geq 0,\  1-\frac{\epsilon_m}{\epsilon_n} \leq |x| \leq 1,  \text{ and } 0\leq |y| \leq \frac{\epsilon_n}{\epsilon_m}(|x|-1) +1 \\
&\text{Type III'} & \quad & xy \geq 0,\  1-\frac{\epsilon_m}{\epsilon_n} \leq |x| \leq 1,  \text{ and } \frac{\epsilon_n}{\epsilon_m}(|x|-1)+1 \leq |y| \leq 1 \\
&\text{Type IV'} & \quad & xy \leq 0,\  1-\frac{\epsilon_m}{\epsilon_n} \leq |x| \leq 1,  \text{ and } \frac{\epsilon_n}{\epsilon_m}(|x|-1)+1 \leq |y| \leq 1;  \text{ or } |x| \leq 1-\frac{\epsilon_m}{\epsilon_n} \\
\end{align*}
The mapping $\psi_n^m$ is almost identical to the map $\phi_n^m$ except the rectangular boundary of the $I_m \times I_n$ cell is twisted in the clockwise direction instead of the counter-clockwise direction (Figure \ref{F3}).  Again note that $\rho(\Psi_n^m, \text{id}) \leq 3 \epsilon_m$.
\begin{figure}
  \centering
  \includegraphics[scale=0.7]{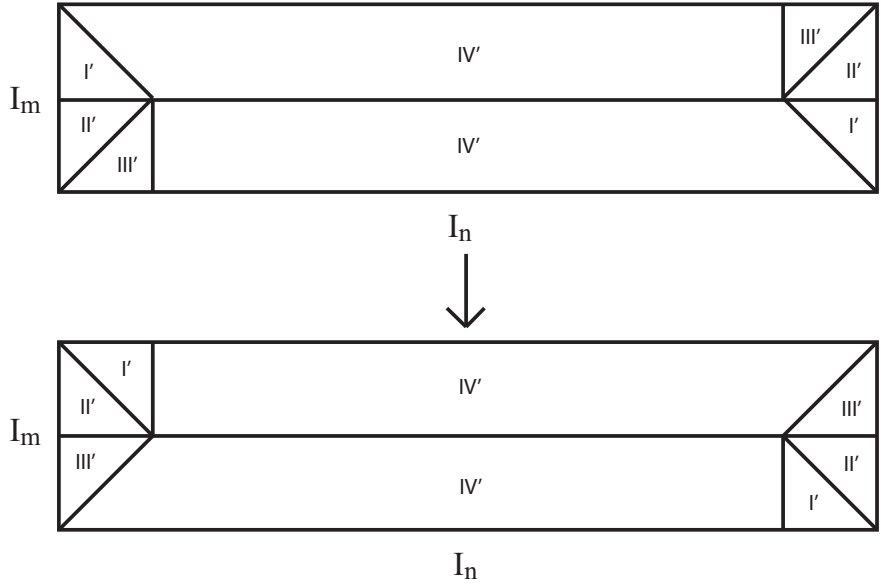}
  \caption{A guide to the mapping $\psi_n^m$.}
  \label{F3}
\end{figure}
By Theorem \ref{T4}, if a sequence of compositions of homeomorphisms is chosen carefully enough, then the sequence will converge to a homeomorphism.  In fact, it is sufficient that the sequence be selected so that it is Cauchy in $\zeta$.  Then in order to select a sequence of homeomorphisms $$\left\{ \Phi_{n(i)}^{m(i)}\Phi_{n(i-1)}^{m(i-1)}\ldots \Phi_{n(1)}^{m(1)} \right\}$$ which converges to a homeomorphism, given an $\epsilon > 0$, it must be possible to choose $m_{i+1}$ so that $$ \zeta\left( \Phi_{n(i+1)}^{m(i+1)}\Phi_{n(i)}^{m(i)}\Phi_{n(i-1)}^{m(i-1)}\ldots \Phi_{n(1)}^{m(1)}, \Phi_{n(i)}^{m(i)}\Phi_{n(i-1)}^{m(i-1)}\ldots \Phi_{n(1)}^{m(1)} \right) < \epsilon.$$
This may be accomplished if the following two conditions are satisfied:
\begin{align*}
& \rho \left( \Phi_{n(i+1)}^{m(i+1)}\Phi_{n(i)}^{m(i)}\Phi_{n(i-1)}^{m(i-1)}\ldots \Phi_{n(1)}^{m(1)}, \Phi_{n(i)}^{m(i)}\Phi_{n(i-1)}^{m(i-1)}\ldots \Phi_{n(1)}^{m(1)} \right) < \frac{\epsilon}2 \\
& \rho \left( \Psi_{n(1)}^{m(1)}\Psi_{n(2)}^{m(2)}\ldots \Psi_{n(i)}^{m(i)}\Psi_{n(i+1)}^{m(i+1)}, \Psi_{n(1)}^{m(1)}\Psi_{n(2)}^{m(2)}\ldots \Psi_{n(i)}^{m(i)} \right) \leq \frac{\epsilon}2
\end{align*}
The first condition is easy to satisfy.  Since $\rho\left(\Phi^{m+1}_{n+1}, \text{id}\right) \leq 3 \epsilon_{m(i+1)}$, we simply require that $\epsilon_{m(i+1)} < \frac{\epsilon}6$.  The second condition requires more analysis.  Since $\Psi_n^m$ moves points using the same type of twist as $\Phi_n^m$ only in the opposite direction, it is also true that $\rho(\Psi_n^m, \text{id}) \leq 3 \epsilon_m$.  However, the difficulty in satisfying the second condition is that although a point $p$ may be very close to $\Psi^{m(i+1)}_{n(i+1)}(p)$, these points may not be close when applied to $\Psi_{n(1)}^{m(1)}\Psi_{n(2)}^{m(2)}\ldots \Psi_{n(i)}^{m(i)}$.  But, if there is a constant $K$ satisfying the condition $$d\left(\Psi_n^m(p), \Psi_n^m(q)\right) \leq K d(p,q)$$ for any two points $p,q \in Q$, then we have $$d \left( \Psi_{n(1)}^{m(1)}\Psi_{n(2)}^{m(2)}\ldots \Psi_{n(i)}^{m(i)}\Psi_{n(i+1)}^{m(i+1)}(p), \Psi_{n(1)}^{m(1)}\Psi_{n(2)}^{m(2)}\ldots \Psi_{n(i)}^{m(i)}(p) \right) \leq K^i d\left(\Psi_{n(i+1)}^{m(i+1)}(p),p \right).$$  Applied to the sup-norm metric $$\rho \left( \Psi_{n(1)}^{m(1)}\Psi_{n(2)}^{m(2)}\ldots \Psi_{n(i)}^{m(i)}\Psi_{n(i+1)}^{m(i+1)}, \Psi_{n(1)}^{m(1)}\Psi_{n(2)}^{m(2)}\ldots \Psi_{n(i)}^{m(i)} \right) \leq K^i \rho\left(\Psi_{n(i+1)}^{m(i+1)},\text{id}\right).$$  By also choosing $m_{i+1}$ so that $\rho\left(\Psi_{n(i+1)}^{m(i+1)}, \text{id}\right) \leq \frac{\epsilon}{2K^i}$, we are able to obtain the desired bound.

We now argue that such a constant exists, and in fact that $K = 2^3$.  Given $p, q \in Q$, the change in separation of these two points applied to $\Psi_n^m$ depends entirely on how the points are moved by $\overline{\left[ \psi_n^m \right]}^3$ in the $I_m \times I_n$ cell.  Let $\overline{p} = (p_n,p_m)$ and $\overline{q} = (q_n,q_m)$ and define $\overline{d}:(I_m \times I_n) \times (I_m \times I_n) \to \mathbb{R}$; $$\overline{d}(\overline{p}, \overline{q}) = \epsilon_n | p_n - q_n | + \epsilon_m | p_m - q_m |.$$  Careful observation reveals that the distance $\overline{d}(\overline{p},\overline{q})$ is realized by the length of some path consisting of horizontal and vertical segments (Figure \ref{F4}).
\begin{figure}
  \centering
  \includegraphics[scale=0.7]{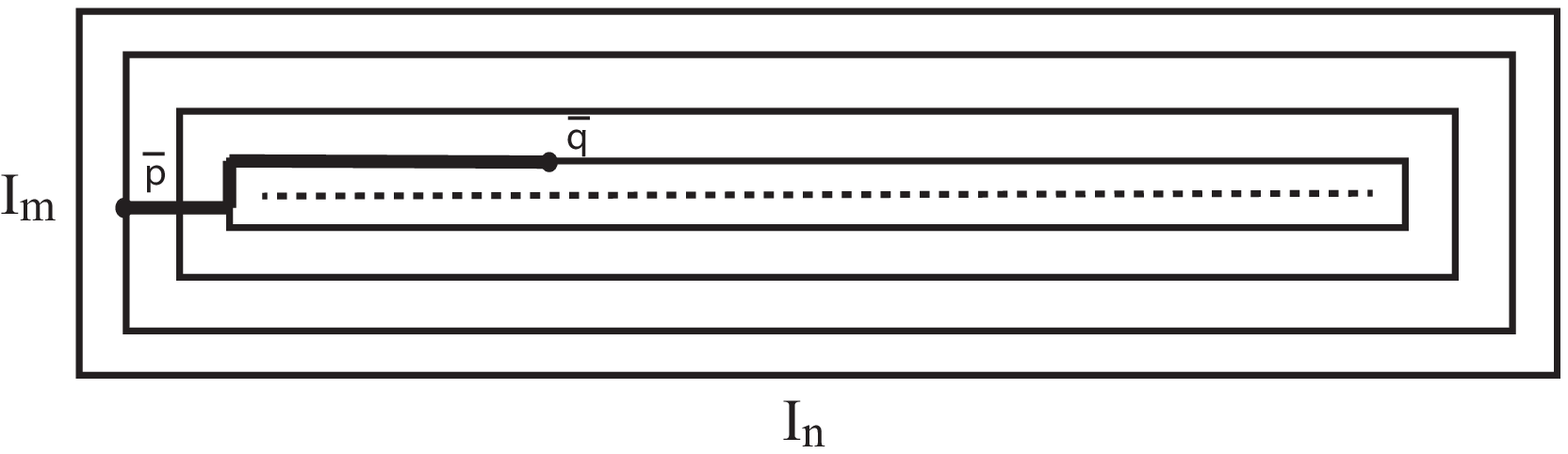}
  \caption{The $\overline{d}$ distance between $\overline{p}$ and $\overline{q}$.}
  \label{F4}
\end{figure}
Segments which lie on the boundary of ``concentric'' rectangles are not changed in length by $\overline{\psi_n^m}$.  Subsegments which form right angles between boundaries are shifted by $\frac{\pi}4$, separating the endpoints by twice the original distance (Figure \ref{F5}).
\begin{figure}
  \centering
  \includegraphics[scale=0.7]{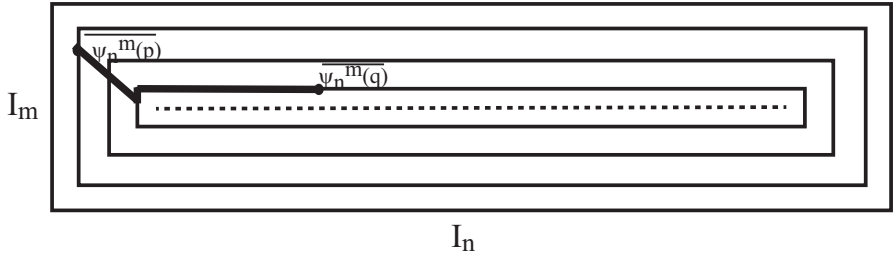}
  \caption{The $\overline{d}$  distance between $\overline{\psi^m_n(p)}$ and $\overline{\psi^m_n(q)}$.}
  \label{F5}
\end{figure}
If $\gamma$ is the contribution to the distance between $p$ and $q$ on all other factors, then $$d(p,q) = \overline{d}(\overline{p}, \overline{q}) + \gamma$$ and
\begin{align*}
d(\psi_n^m(p),\psi_n^m(q)) =& \overline{d}(\overline{p}, \overline{q}) + \gamma \\
\leq & 2 d(p,q).
\end{align*}
Thus  $$d(\psi_n^m(p),\psi_n^m(q)) \leq  2^3 d(p,q).$$  Therefore $K = 2^3$.

As discussed previously, $m_{i+1}$ should be chosen so that $\rho\left(\Phi^{m(i+1)}_{n(i+1)}, \text{id}\right) \leq \frac{\epsilon}{2 K^i}$.   In other words we require that $3  \epsilon_{m(i+1)} \leq \frac{\epsilon}{2^{3i+1}}$.  Then both conditions (1) and (2) are satisfied by choosing $m_{i+1}$ so that $\epsilon_{m(i+1)}\leq \frac{\epsilon}{3\cdot 2^{3i+1}}$.  Noting that $\epsilon_{m(i+1)} = \frac{1}{2^{m(i+1)}}$ the inequality reduces to $m_{i+1} \geq \log_2\left(\frac 3{\epsilon}\right) + (3i+1)$.  This gives the desired result, which is $$\zeta \left( \Phi_{n(i+1)}^{m(i+1)}\Phi_{n(i)}^{m(i)}\Phi_{n(i-1)}^{m(i-1)}\ldots \Phi_{n(1)}^{m(1)}, \Phi_{n(i)}^{m(i)}\Phi_{n(i-1)}^{m(i-1)}\ldots \Phi_{n(1)}^{m(1)} \right) < \epsilon.$$

\begin{thm} \label{T5}
For any point $p \in B(Q)$, there is a homeomorphism $h$ such that $h(p)$ is contained in the pseudo-interior of $Q$.
\end{thm}

\begin{proof}
Let $p$ be the point of the pseudo-boundary of $Q$.  We would like to choose a sequence of homeomorphisms from the family of homeomorphisms, described above, which removes $p$ from the pseudo-boundary and whose composition converges to a homeomorphism.

First, we need to guarantee that given a sequence of the compositions of homeomorphisms of the form $\Phi_n^m$, that the sequence will converge to a homeomorphism.  Since $\rho(\Phi_n^m, \text{id})$ depends only on $m$, then by Theorem \ref{T4} we may choose each $m_i$ large enough so that $$h = \underset{n \to \infty}{\lim} \Phi_{n(i)}^{m(i)}\Phi_{n(i-1)}^{m(i-1)}\ldots \Phi_{n(1)}^{m(1)} $$ converges to a homeomorphism where $\{n_i\}$ is any increasing sequence and $\{m_i\}$ is an increasing sequence such that $m_i > n_i$ for all $i \in J$.  Letting $\epsilon = \frac{3}{2^{i+3}}$ in the previous argument we can see it is sufficient to choose $\{m_i\}$ as a subsequence of $\{4i\}$.

Let $\{ j_1, j_2, \ldots \}$ be the ordered set of all index numbers for which $p_{j_i} \in [-1,1]$.  Let $n_1 = j_1$.  Choose $m_1 \in \{4i\}$ such that $m_1 > n_1$.  For $k \geq 2$, we inductively define $n_{k}$ and $m_{k}$  as follows: Let $n_{k} = \min \left( \{m_1,\ldots, m_{k-1}, j_1,j_2, \ldots \} - \{n_1, \ldots, n_{k-1}\}\right)$. Choose $m_{k} \in \{ 4i \}$ such that $m_{k} > m_{k-1}$ and $m_{k} > n_{k}$.  Then
\begin{equation} \label{E2}
h = \underset{n \to \infty}{\lim} \Phi_{n(i)}^{m(i)}\Phi_{n(i-1)}^{m(i-1)}\ldots \Phi_{n(1)}^{m(1)}
\end{equation}
defines a homeomorphism.  Furthermore, at each $i$th stage, $$\Phi_{n(i)}^{m(i)}\Phi_{n(i-1)}^{m(i-1)}\ldots \Phi_{n(1)}^{m(1)}(p) \subset \underset{i=1}{\overset{n(i)}{\Pi}}I_i^o \times \underset{j={n(i)}}{\overset{\infty}{\Pi} } I_j.$$  Thus in the limit, $p$ is mapped into $\underset{i=1}{\overset{\infty}{\Pi}} I_i^o$, so $h(p)$ is in the pseudo-interior of $Q$.
\end{proof}

\begin{thm}
The Hilbert cube $Q$ is homogeneous.  Moreover, for points $p, q \in Q$, a homeomorphism that is the composition of maps of the form Equation \ref{E1} of the introduction and of the form Equation \ref{E2} of the proof of Theorem \ref{T5} can be defined to realize a self-homeomorphism of $Q$ mapping $p$ to $q$.
\end{thm}

\begin{proof}
Let $p,q \in Q$.  Let $f$ denote a homeomorphism of the type defined in Equation \ref{E1} of the introduction. Recall that $f$ can be defined to map any specified pseudo-interior point to any other specified pseudo-interior point, while fixing the pseudo-boundary of $Q$.  Let $h$ or $h'$ denote a homeomorphism of the type defined in Equation \ref{E2} of the proof of Theorem \ref{T5}.  Then $h$ or $h'$ can be defined to  map any specified pseudo-boundary point to some non-specified pseudo-interior point.
\begin{enumerate}
\item[Case 1.] If both $p$ and $q$ are interior points, a homeomorphism of the form  $f$ maps $p$ to $q$.
\item[Case 2.] If $p$ is a boundary point and $q$ is an interior point, a homeomorphism of the form $fh$ maps $p$ to $q$.
\item[Case 3.] If $p$ is an interior point and $q$ is a boundary point, a homeomorphism of the form $h^{-1}f$ maps $p$ to $q$.
\item[Case 4.] If both $p$ and $q$ are boundary points, a homeomorphism of the form  $h^{-1}fh'$ maps $p$ to $q$.
\end{enumerate}
Thus in all cases, there is a self-homeomorphism of $Q$ mapping $p$ to $q$ of the desired form.
\end{proof}

\end{document}